\documentclass[11pt]{amsart}
\usepackage[dvipdfmx]{graphicx, color}
\usepackage{amssymb, amsmath,amsthm}
\usepackage{epstopdf}
\usepackage{comment}
\newcommand{\nc}{\newcommand}
\nc{\dmo}{\DeclareMathOperator}
\nc{\nt}{\newtheorem}

\nt{theorem}{Theorem}
\nt{lemma}{Lemma}

\newtheorem{thm}{{\bf Theorem}}
\newtheorem{lem}[thm]{{\bf Lemma}}

\newtheorem{rem}[thm]{Remark}

\newtheorem{ques}[thm]{Question}

\numberwithin{equation}{section}

\begin{document}

\title[On hyperbolic surface bundles over $S^1$ as  branched double covers of $S^3$]
{On hyperbolic surface bundles over the circle as  branched double covers of the 
$3$-sphere}

\author[S. Hirose]{%
Susumu Hirose
}
\address{%
Department of Mathematics,  
Faculty of Science and Technology, 
Tokyo University of Science, 
Noda, Chiba, 278-8510, Japan}
\email{%
hirose\_susumu@ma.noda.tus.ac.jp
}

\author[E. Kin]{%
    Eiko Kin
}
\address{%
       Department of Mathematics, Graduate School of Science, Osaka University Toyonaka, Osaka 560-0043, JAPAN
}
\email{%
        kin@math.sci.osaka-u.ac.jp
}



\date{
August 8, 2019
}

	
\begin{abstract} 
The branched virtual fibering theorem by Sakuma states that
 every closed orientable $3$-manifold with a Heegaard surface of genus $g$ 
 has a branched double cover  which is a genus $g$ surface bundle over the circle. 
 It is proved by Brooks that such a surface bundle can be chosen to be hyperbolic. 
We prove that 
 the minimal entropy over all hyperbolic,  genus $g$ surface bundles 
 as  branched double covers of the $3$-sphere 
 behaves like 1/$g$. 
 We also give an alternative construction of surface bundles over the circle in Sakuma's theorem 
 when 
closed $3$-manifolds are branched double covers of the $3$-sphere branched over links. 
A feature of surface bundles coming from our construction  is that 
the monodromies   can be read off 
the braids obtained from the links as the branched set. 
\end{abstract}
\maketitle

\section{Introduction}
\label{section_Introduction}

This paper concerns the {\it branched} virtual fibering theorem by Makoto Sakuma. 
To state his theorem 
let $\Sigma = \Sigma_{g,p}$ be an orientable, 
connected surface of genus $g$ with $p$ punctures possibly $p= 0$, 
and let us set $\Sigma_g= \Sigma_ {g,0}$. 
The mapping class group $\mathrm{Mod}(\Sigma)$  is 
the group of isotopy classes of orientation-preserving 
self-homeomorphisms on $\Sigma$ which preserve the punctures setwise. 
By Nielsen-Thurston classification, 
elements  in $\mathrm{Mod}(\Sigma)$ fall into three types: 
periodic, reducible, pseudo-Anosov \cite{Thurston88}. 
To each pseudo-Anosov element $\phi$, 
there is an associated dilatation (stretch factor) $\lambda(\phi)>1$~(see \cite{FM12} for example). 
We call the logarithm of the dilatation $\log(\lambda(\phi))$  the {\it entropy} of $\phi$. 

Choosing a representative 
$f : \Sigma \to \Sigma$ of $\phi$ 
we define the {\it mapping torus\/} $T_{\phi}$ by 
$$
T_{\phi} = \Sigma \times \mathbb{R} / \sim, 
$$
where $(x,t) \sim (f(x), t+1)$ for $x \in \Sigma$, $t \in \mathbb{R}$. 
We call $\Sigma$ the {\it fiber surface} of $T_{\phi}$.  
The $3$-manifold $T_{\phi}$ is a $\Sigma$-bundle over the circle with the monodromy $\phi$. 
By Thurston \cite{Thurston98} 
$T_{\phi}$ admits a hyperbolic structure of finite volume 
if and only if $\phi$ is pseudo-Anosov.

The following theorem  is due to Sakuma~\cite[Addendum~1]{Sakuma81}. 
See also \cite[Section~3]{Brooks85}.

\begin{thm}[Branched virtual fibering theorem]
\label{thm_Sakuma}
Let $M$ be a closed orientable $3$-manifold. 
Suppose that $M$ admits a genus $g$ Heegaard splitting. 
Then there is a $2$-fold branched cover $\widetilde{M}$ of $M$ 
which is a $\Sigma_g$-bundle over the circle. 
\end{thm}

It is proved by Brooks~\cite{Brooks85} that  $\widetilde{M}$ in Theorem~\ref{thm_Sakuma} can  be chosen to be hyperbolic if  
$g \ge \max(2, g(M))$, 
where $g(M)$ is the Heegaard genus of $M$. 
See also \cite{Montesinos87} by Montesinos.

Let $D_g(M)$ be the subset of $\mathrm{Mod}(\Sigma_g)$ consisting of elements 
$\phi$ such that $T_{\phi}$ is homeomorphic to a $2$-fold branched cover of $M$ branched over some link. 
By Theorem~\ref{thm_Sakuma} we have $D_g(M) \ne \emptyset$. 
By Brooks together with the stabilization of Heegaard splittings, 
there is a pseudo-Anosov element in $D_g(M) $ for each $g \ge \max(2, g(M))$. 
The set of fibered $3$-manifolds 
$T_{\phi}$ over all $\phi \in D_g(M)$ 
possesses various properties inherited under branched covers of $M$. 
It is natural to ask about the dynamics of pseudo-Anosov elements in $D_g(M)$. 
We are interested in the set of entropies of pseudo-Anosov mapping classes.

We fix a surface $\Sigma$ and consider 
the set of entropies 
$$\{\log \lambda(\phi)\ |\ \phi \in \mathrm{Mod}(\Sigma)\  \mbox{is pseudo-Anosov}\}$$ 
which is a closed, discrete subset of ${\Bbb R}$ (\cite{AY81}). 
For any subset $G \subset \mathrm{Mod}(\Sigma)$ 
let $\delta(G)$ denote the minimum of dilatations $\lambda(\phi)$ over all pseudo-Anosov elements $\phi \in G$. 
Then $\delta(G) \ge \delta(\mathrm{Mod}(\Sigma))$. 
For real valued functions $f$ and $h$,  we write $f \asymp h$ 
 if there is a universal constant $c$  such that $h/c \le f \le c h$. 
It is proved by Penner~\cite{Penner91} that 
 $$\log \delta(\mathrm{Mod}(\Sigma_g)) \asymp \dfrac{1}{g}.$$ 
A question arises: 
what can we say about the asymptotic behavior of the minimal entropies 
$\log \delta(D_g(M))$'s  for each closed $3$-manifold $M$? 
In this paper we consider this question when $M$ is the $3$-sphere $S^3$. 
Our main theorem is the following. 

\begin{theorem}\label{thm_main}
We have $\log \delta(D_g(S^3)) \asymp \dfrac{1}{g}$. 
\end{theorem}

\begin{center}
\begin{figure}[t]
\includegraphics[height=4.5cm]{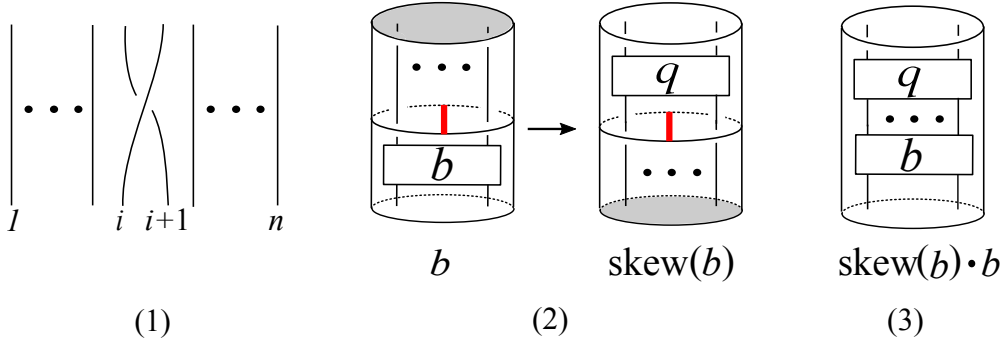}
\caption{(1) $\sigma_i \in B_n$. 
(2) Involution on the cylinder. 
Thick segment in the middle is the fixed point set of the involution. 
(3) $\widetilde{b}= \operatorname{skew}(b) \cdot b$ is invariant under such an involution.} 
\label{fig:braids}
\end{figure}
\end{center}

Let $q_L: M_L \rightarrow S^3$ denote the $2$-fold branched covering map of $S^3$ 
branched over a link $L$ in $S^3$. 
Along the way in the proof of Theorem~\ref{thm_main} 
we give an alternative proof of Theorem \ref{thm_Sakuma} when $M=M_L$ in Theorem \ref{thm:2-fold}. 
A feature of surface bundles $\widetilde{M_L}$ coming from  our construction is that 
their monodromies   can be read off 
the braids obtained from the links as the branched set. 
To state Theorem \ref{thm:2-fold}, we need 3 ingredients. 
\medskip

\noindent
{\bf 1. Involution $\operatorname{skew}: B_n \to B_n$.} 
Let $B_n$ be the (planar) braid group with $n$ strands 
and let $\sigma_i$ denote the Artin generator of $B_n$ as  in Figure \ref{fig:braids}(1). 
We define an involution 
\begin{eqnarray*}
\operatorname{skew}: B_n &\to& B_n
\\
 \sigma_{n_1}^{\epsilon_1} \sigma_{n_2}^{\epsilon_2} \cdots \sigma_{n_k}^{\epsilon_k} 
 &\mapsto& 
 \sigma_{n-n_k}^{\epsilon_k} \cdots \sigma_{n- n_2}^{\epsilon_2} \sigma_{n-n_1}^{\epsilon_1}, 
\hspace{5mm} \epsilon_i = \pm 1.
\end{eqnarray*}
We say that $b \in B_n$ is {\it skew-palindromic} if $\operatorname{skew}(b)= b$. 
The braid  $\operatorname{skew}(b) \cdot b$ is skew-palindromic  for any $b \in B_n$. 
(There is a skew-palindromic braid which can not be written by 
$\operatorname{skew}(b) \cdot b$ for any $b$, for example $\sigma_1 \sigma_2 \sigma_3 \in B_4$.) 
We write $$\widetilde{b} = \operatorname{skew}(b) \cdot b.$$  
Note that $\operatorname{skew}: B_n \rightarrow B_n$ is induced by the involution on the cylinder as shown in Figure \ref{fig:braids}(2) 
and skew-palindromic braids are invariant under such an involution.

In the later section, the map $\operatorname{skew}$ is also regarded as a map 
from the braid group on the sphere or the annulus to itself. 
The above assertion for the braid $\widetilde{b} = \operatorname{skew}(b) \cdot b$ 
also holds in this setting.

\begin{center}
\begin{figure}[t]
\includegraphics[height=2cm]{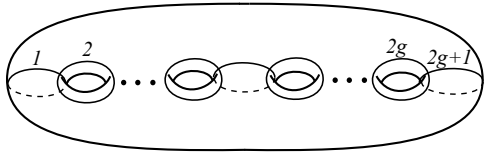}
\caption{Simple closed curves  on $\Sigma_g$.}
\label{fig:hyperelliptic}
\end{figure}
\end{center}


\noindent
{\bf 2. Homomorphism $\mathfrak{t}: B_{2g+2} \rightarrow \mathrm{Mod}(\Sigma_g)$.} 
Let $t_i$ denote the right-handed Dehn twist 
about the simple closed curve with the number $i$ in Figure \ref{fig:hyperelliptic}. 
Then there is a homomorphism 
$$\mathfrak{t}: B_{2g+2} \rightarrow \mathrm{Mod}(\Sigma_g)$$ 
which sends $\sigma_i$ to $t_i$ for $i=1, \ldots, 2g+1$, 
since $\mathrm{Mod}(\Sigma_g)$ has the braid relation. 
(We apply elements of mapping class groups from right to left.) 
The {\it hyperelliptic mapping class group} $\mathcal{H}(\Sigma_g)$ 
is the subgroup of $\mathrm{Mod}(\Sigma_g)$ consisting of elements with representative homeomorphisms 
that commute with some fixed hyperelliptic involution. 
By Birman-Hilden \cite{BH71}, $\mathcal{H}(\Sigma_g)$ is generated by $t_i$'s. 
Thus 
$$\mathcal{H}(\Sigma_g)= \mathfrak{t}(B_{2g+2}) .$$

\noindent
{\bf 3. Circular plat closure $C(b)$.} 
We use two types of links in $S^3$ obtained from braids. 
One is the {\it closure} $\mathrm{cl}(\beta)$ of  $\beta$ $\in B_{g+1}$ 
as in Figure \ref{fig:circular-plat}(1). 
The other is the {\it circular plat closure} $ C(b)$ of  $b \in B_{2g+2}$ with even strands as in Figure \ref{fig:circular-plat}(2)(3). 
We also use the link $C(b) \cup W$, the union of $C(b)$ and the trivial link $W= O \cup O'$ with $2$ components 
as shown in Figure~\ref{fig:circular-plat}(4).

\begin{center}
\begin{figure}[t]
\includegraphics[height=4.5cm]{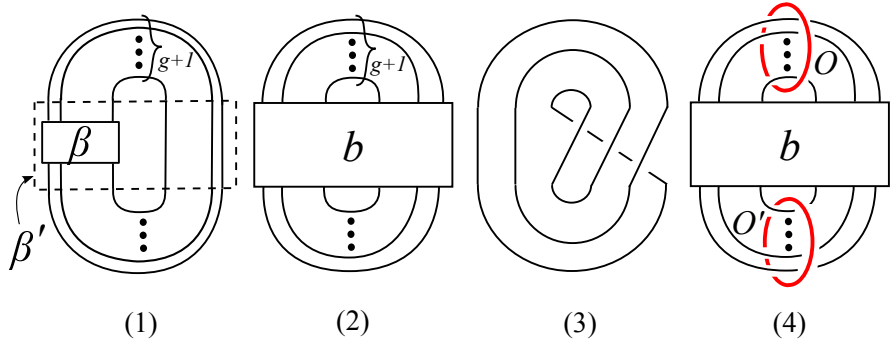}
\caption{(1) $\mathrm{cl}(\beta)$ for $\beta \in B_{g+1}$ and $\beta' \in B_{2g+2}$. 
(2) $C(b)$ for $b \in B_{2g+2}$. 
(3) $C(\sigma_3 \sigma_4 \sigma_5)$. 
(4) $C(b) \cup W$, where $W= O \cup O'$.}
\label{fig:circular-plat}
\end{figure}
\end{center}

Any link  in $S^3$ can be represented by $C(\beta')$ for some braid $\beta'$. 
To see this, it is well-known that 
$L$ is  the closure $\mathrm{cl}(\beta)$ for some $\beta \in B_{g+1}$ ($g \ge 1$).  
The desired braid $\beta' \in B_{2g+2}$ can be  obtained from $\beta$ 
by adding  $g+1$ straight strands as in Figure \ref{fig:circular-plat}(1). 

For a braid  $b \in B_{2g+2}$ 
let $q= q_{C(b)}: M_{C(b)} \rightarrow S^3$ be the $2$-fold branched covering map of $S^3$ 
branched over  $C(b)$. 
There is a $(g+1)$-bridge sphere $S$ for the link $C(b) \subset S^3$. 
Hence $M_{C(b)}$ admits a genus $g$ Heegaard splitting with the 
Heegaard surface $q^{-1}(S)$.  
Then we have the following result.

\begin{theorem}\label{thm:2-fold}
Let $ \widetilde{M_{C(b)}}$ be the $2$-fold branched cover of 
$M_{C(b)}$ branched over the link $q^{-1}(W)$. Then 
$\widetilde{M_{C(b)}}$ is homeomorphic to the mapping torus $T_{\mathfrak{t}(\widetilde{b})}$. 
\end{theorem}

\begin{rem}
\label{rem:after_thm}
To be precise, $ \widetilde{M_{C(b)}}$ is the $2$-fold branched cover of $M_{C(b)}$ 
branched over $q^{-1}(W)$ obtained as the ${\Bbb Z}/2 {\Bbb Z}+ {\Bbb Z}/2 {\Bbb Z}$-cover of $S^3$ 
branched over the link $C(b) \cup W$ associated with the epimorphism 
$H_1(S^3 \setminus (C(b) \cup W)) \rightarrow {\Bbb Z}/2 {\Bbb Z}+ {\Bbb Z}/2 {\Bbb Z}$ 
which maps the meridians of $C(b)$ to the generator of the first factor and the meridians of $W$ 
to the generator of the second factor. 
\end{rem}

\noindent
{\bf Acknowledgments.} 
We thank Makoto Sakuma and Yuya Koda for helpful conversations and comments. 
We also thank the anonymous referee for valuable comments. 
The first author was supported by 
Grant-in-Aid for
Scientific Research (C) (No. 16K05156), JSPS. 
The second author was supported by 
Grant-in-Aid for
Scientific Research (C) (No. 18K03299), JSPS.

\section{Proof of Theorem \ref{thm:2-fold}}
\label{section_reprove}


\begin{proof}[Proof of Theorem~\ref{thm:2-fold}] 
We construct the $3$-sphere $S^3$ 
from two copies of the  $3$-ball $B^3$  
by gluing their boundaries together. 
Consider the link $C(b) \cup W$ 
so that $C(b)$ is contained in one of the $3$-balls, 
and  $W$ is given by the union of the four thick segments in the two $3$-balls, 
see Figure~\ref{fig:quotient}(2). 
Let $S$ be the sphere in $S^3$ which is the union of the two shaded disks 
in the same figure.
The $2$-sphere $S$ is a  $(g+1)$-bridge sphere for  $C(b)$, and 
the preimage $q^{-1}(S)$ 
  is a genus $g$ Heegaard surface of $M_{C(b)}$.

Let $q_{W} : M_W \to S^3$ be  
the $2$-fold branched covering map of $S^3$ branched over $W$ (Figure~\ref{fig:quotient}(1)). 
The preimage  $q^{-1}_{W}(B^3)$  is homeomorphic to the solid torus $D^2 \times S^1$. 
Then $M_W$ is obtained from  two copies of $D^2 \times S^1$ by gluing their boundaries together, 
and hence $M_W$ is homeomorphic to $S^2 \times S^1$. 
Observe that the link 
$q_{W}^{-1} (C(b))$ is the closure of the {\it spherical} braid $\widetilde{b}= \operatorname{skew}(b) \cdot b$, i.e., 
$$ q_{W}^{-1} (C(b))=  \mathrm{cl}(\widetilde{b}) \subset S^2 \times S^1.$$ 

Let $p: N_{\mathrm{cl}(\widetilde{b})} \rightarrow S^2 \times S^1$ 
be the $2$-fold branched covering map of $S^2 \times S^1$ branched over $\mathrm{cl}(\widetilde{b})$. 
The $2$-fold branched cover of the level surface $S^2 \times \{u\}$ for $u \in S^1$ branched at the  $2g+2$ points 
in $\Bigl( (S^2 \times \{u\}) \cap \mathrm{cl}(\widetilde{b}) \Bigr)$ is a closed surface of genus $g$. 
Thus  $N_{\mathrm{cl}(\widetilde{b})}$ is a $\Sigma_g$-bundle over $S^1$ 
with the monodromy $\mathfrak{t}(\widetilde{b}) $, i.e., 
$N_{\mathrm{cl}(\widetilde{b})}$ is homeomorphic to $T_{\mathfrak{t}(\widetilde{b})}$.

We can observe that the composition 
$$q_W \circ p: T_{\mathfrak{t}(\widetilde{b})} \cong N_{\mathrm{cl}(\widetilde{b})} \rightarrow S^3$$ 
is the ${\Bbb Z}/2 {\Bbb Z}+ {\Bbb Z}/2 {\Bbb Z}$-cover of $S^3$ branched over the link 
$C(b) \cup W$ described in Remark \ref{rem:after_thm}. 
Hence $T_{\mathfrak{t}(\widetilde{b})} \cong N_{\mathrm{cl}(\widetilde{b})} $ is identified with 
$2$-fold branched cover  $\widetilde{M_{C(b)}}$. 
This completes the proof.
\end{proof}

\begin{center}
\begin{figure}[t]
\includegraphics[width=4.8 in]{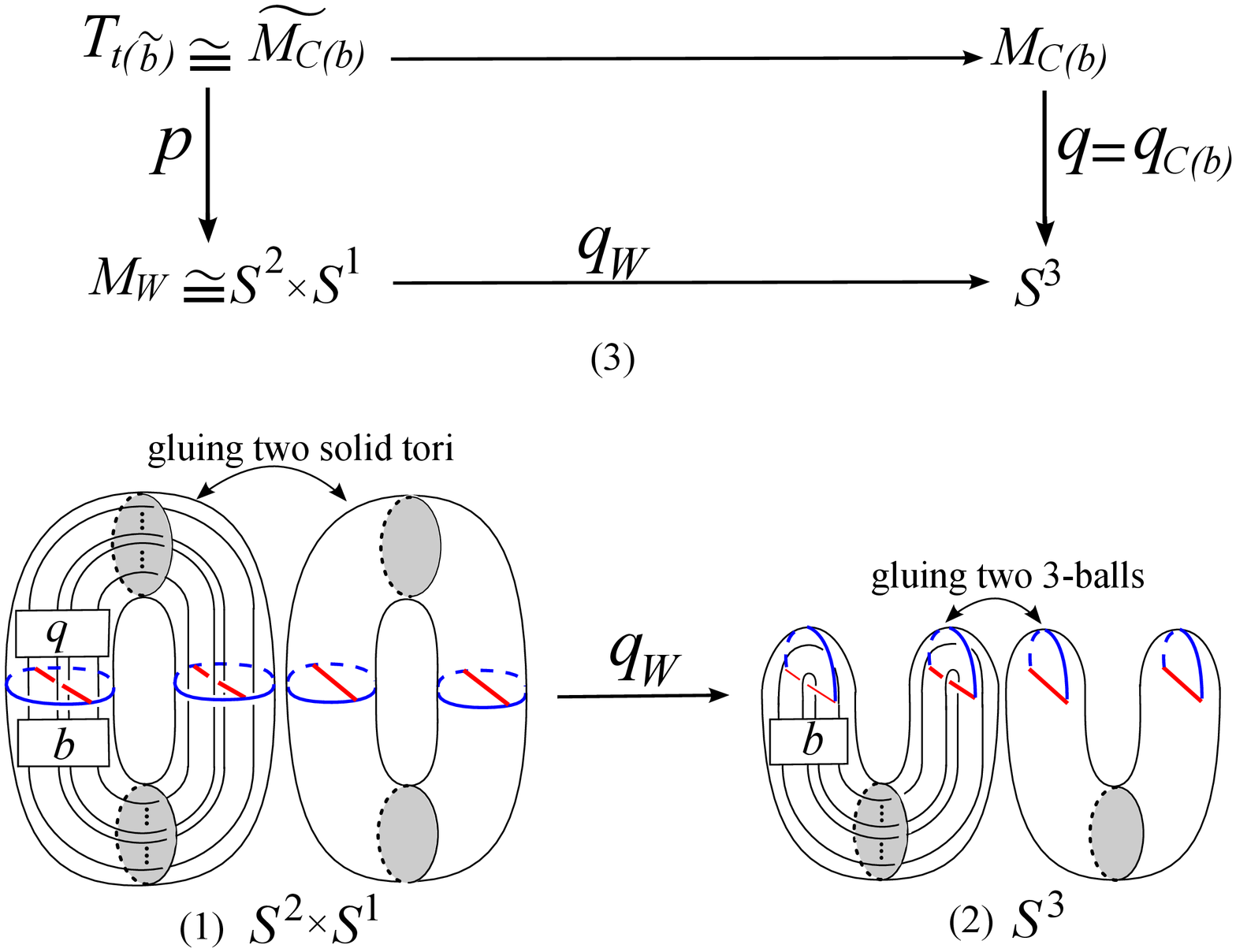}
\caption{
(1) $\mathrm{cl}(\widetilde{b}) \subset  S^2 \times S^1$. 
(2) $C(b) \cup W \subset S^3$. 
(3) Diagram: 
$\widetilde{M_{C(b)}} \rightarrow M_{C(b)}$ is the $2$-fold branched cover of 
$M_{C(b)}$ branched over $q^{-1}(W)$.}
\label{fig:quotient}
\end{figure}
\end{center}

\section{Proof of Theorem \ref{thm_main}}
\label{section_pA}

\begin{center}
\begin{figure}[t]
\includegraphics[width=4 in]{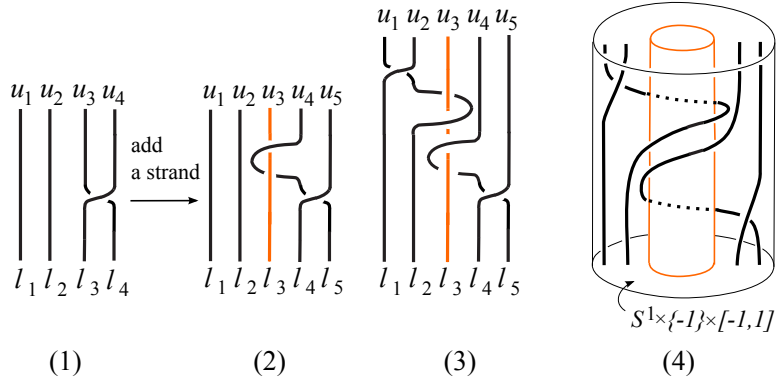}
\caption{
(1) $b= \sigma_3 $. 
(2) $b^{\circ}= \sigma_3^2 \sigma_4 $. 
(3) $\widetilde{b^{\circ}}$. 
($\widetilde{b}= \sigma_1 \sigma_3$ is obtained from $\widetilde{b^{\circ}}  $ by removing the strand 
with endpoints $\ell_3$ and $u_3$.)  
(4) Braid $\widetilde{b} $ on  $\mathcal{A}$.}
\label{fig:adding}
\end{figure}
\end{center}

Given a braid $b$ 
we first give a construction of a braid $b'$ (with more strands than $b$) such that 
$C(b)$ is ambient isotopic  to $C(b')$.

The bottom and  top endpoints of a planar braid with $n$ strands are denoted by 
$l_1, \ldots, l_n$ and $u_1, \ldots, u_n$ from left to right. 
For a braid $b \in B_{2g+2}$  with even strands,  
we choose a braid $b^{\circ} \in B_{2g+3}$ 
obtained from $b$ by adding a strand, say $b^{\circ}(g+2)$ 
connecting the middle of the two points $l_{g+1}$ and $l_{g+2}$  with 
the middle of the two points $u_{g+1}$ and $u_{g+2}$. 
Of course $b^{\circ}$ is not unique. 
For example 
when $b= b_1= \sigma_3 \in B_4$, one can choose  $b^{\circ}(= b_1^{\circ})= \sigma_3^2 \sigma_4 \in B_5$. 
See Figure~\ref{fig:adding}. 

We consider $\widetilde{b^{\circ}} = \operatorname{skew}(b^{\circ}) \cdot b^{\circ} \in B_{2g+3}$ 
with bottom endpoints 
$l_1, \ldots, l_{2g+3}$ and top endpoints $u_1, \ldots, u_{2g+3}$. 
The braid $\widetilde{b^{\circ}} $ 
has the strand $\widetilde{b^{\circ}}(g+2)$ 
with endpoints $l_{g+2}$ and $u_{g+2}$. 
If we remove this strand from $\widetilde{b^{\circ}}$, 
then we obtain $\widetilde{b} = \operatorname{skew}(b) \cdot b$. 

\begin{center}
\begin{figure}[t]
\includegraphics[width=3.5 in]{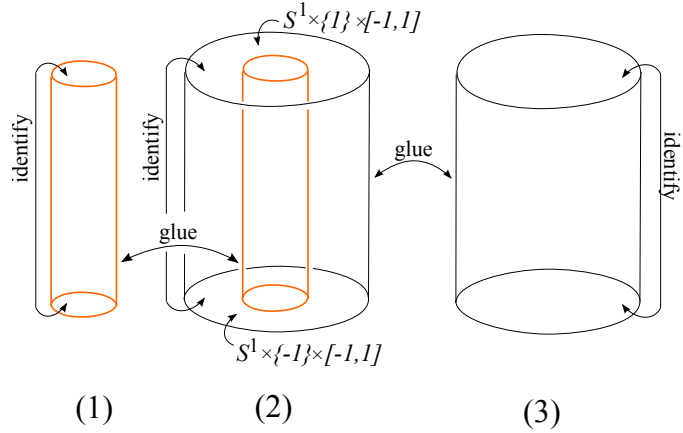}
\caption{
(1) Solid torus $(D^2 \times S^1)_{-1}$. 
(2) $S^1 \times S^1 \times [-1, 1]$ obtained from $S^1 \times [-1,1] \times [-1, 1]$
by identifying the two annuli $S^1 \times \{1\} \times [-1,1]$ and $S^1 \times \{-1\} \times [-1,1]$. 
(3) Solid torus $(D^2 \times S^1)_{+1}$.}
\label{fig:piece}
\end{figure}
\end{center}

Now we construct $S^2 \times S^1$ 
from three pieces, two solid tori $(D^2 \times S^1)_{\pm1}$ 
and the product $S^1 \times S^1 \times [-1, 1]$ 
of a torus $S^1 \times S^1$ and the interval $[-1,1]$  by gluing  
$(\partial D^2 \times S^1)_i$ to $S^1 \times S^1 \times \{ i \}$ 
together for $i = \pm1$. 
See Figure~\ref{fig:piece}. 
We think of $S^{1}$ as the quotient space $[-1,1] / (-1 \sim 1)$, and 
consider the product 
$$S^{1} \times S^{1} \times [-1,1] = S^{1} \times [-1,1]/(-1 \sim 1) \times [-1,1].$$ 
For the braided link $\mathrm{br}(\widetilde{b^{\circ}}) = \mathrm{cl}(\widetilde{b^{\circ}}) \cup A$ in $S^3$, 
we perform the $0$-surgery along the braid axis $A$. 
Then the image of $\mathrm{cl}(\widetilde{b^{\circ}})$ forms a link in $S^2 \times S^1$, 
which continues to denote by the same symbol $\mathrm{cl}(\widetilde{b^{\circ}})$. 
We deform this link $\mathrm{cl}(\widetilde{b^{\circ}}) $ in $S^2 \times S^1$ so that 
the  knot $\mathrm{cl}(\widetilde{b^{\circ}} (g+2))$ becomes  the core of $(D^2 \times S^1)_{-1}$ 
and $\mathrm{cl}(\widetilde{b}) = \mathrm{cl}(\widetilde{b^{\circ}}) \setminus \mathrm{cl}(\widetilde{b^{\circ}} (g+2))$ 
is contained in $ S^1 \times S^1 \times [-1, 1]$. 
One can regard $\widetilde{b}$ as a braid on the annulus 
$\mathcal{A}:= S^{1} \times \{ -1 \} \times [-1,1]$ 
which is embedded in $S^{1} \times [-1,1] \times [-1,1]$, and 
one can think of the link $\mathrm{cl}(\widetilde{b})$ 
as the closure of the braid $\widetilde{b}$ on  $\mathcal{A}$. 
See Figure~\ref{fig:adding}(4). 
Let 
$$R: S^2 \times S^1 \rightarrow S^2 \times S^1$$ 
be the deck transformation of 
$q_{W} : S^2 \times S^1 \to S^3$  in the proof of Theorem~\ref{thm:2-fold}. 
Then 
$q_{W}$ sends the fixed point set of the involution $R$ to the trivial link $W$ (Figure~\ref{fig:quotient}(1)(2)).  
Let 
$$f: S^1 \times S^1 \rightarrow S^1 \times S^1 $$  
be any orientation-preserving homeomorphism. 
We may assume that $f$ commutes with the involution 
\begin{eqnarray*}
\iota : S^1 \times S^1 &\to& S^1 \times S^1
\\
(x,y) &\mapsto& (-x,-y). 
\end{eqnarray*}
We consider the homeomorphism 
$$\Phi_f = f \times \mathrm{id}_{[-1,1]} : S^1 \times S^1 \times [-1,1] \rightarrow S^1 \times S^1 \times [-1,1].$$
The image of $\mathrm{cl}(\widetilde{b})$ under $\Phi_f$ may or may not be the closure of some braid on $\mathcal{A}$. 
We assume the former case $(*)$: 
\begin{quote}
$(*)$ $\Phi_f(\mathrm{cl}(\widetilde{b})) = \mathrm{cl}(\gamma)$ for some braid $\gamma$
on $\mathcal{A}$. 
\end{quote}
Then the involution 
$R|_{S^1 \times S^1 \times [-1, 1]} = \iota \times \mathrm{id}_{[-1,1]}$ has the following property. 
$$R(\mathrm{cl}(\beta)) = \mathrm{cl}(\operatorname{skew}(\beta))\hspace{2mm} \mbox{for any braid }\beta \mbox{ on}\  \mathcal{A}.$$
cf. Figure~\ref{fig:braids}(2). 
Since $f$ commutes with $\iota$, 
it follows that $\Phi_f$  commutes with $R|_{S^1 \times S^1 \times [-1, 1]}$. 
Hence we have 
$$\mathrm{cl}(\gamma) = \Phi_f(\mathrm{cl}(\widetilde{b})) = \Phi_f(\mathrm{cl}(\operatorname{skew}(\widetilde{b}))) = 
\Phi_f R (\mathrm{cl}(\widetilde{b}))= R \Phi_f(\mathrm{cl}(\widetilde{b})) =R ( \mathrm{cl}(\gamma)) .$$ 
(The first and last equality come from the assumption $(*)$, and the second equality holds since $\widetilde{b}$ is skew-palindromic.)   
Thus $\mathrm{cl}(\gamma) = R ( \mathrm{cl}(\gamma)) $. 
We further assume that 
\begin{quote}
$(**)$ $\bigl( \Phi_f(\mathrm{cl}(\widetilde{b}))= \bigr) \mathrm{cl}(\gamma)= \mathrm{cl}(\widetilde{b_f})$ 
for some braid $b_f$ on  $\mathcal{A}$. 
\end{quote}

\begin{rem} 
Clearly  $(**)$ implies 
$\mathrm{cl}(\gamma) = R (\mathrm{cl}(\gamma))$. 
It is likely the converse holds. 
\end{rem}

Now, we think of the braid $b_{f}$ on $\mathcal{A}$ as a planar braid as in Figure \ref{fig:adding}, 
and consider the link $C(b_{f})$ in $S^{3}$. 
We have the following lemma. 

\begin{lem}
\label{lem_disk-twist-inv} 
Under the assumptions $(*)$ and $(**)$, 
$C(b)$ and $C(b_f)$ are ambient isotopic. 
\end{lem}

\begin{proof} 
Note that the quotient $ \big(S^1 \times S^1 \times [-1,1] \big)/R$ is homeomorphic to $S^2 \times [-1,1]$. 
Since $\Phi_f$ commutes with $R|_{S^1 \times S^1 \times [-1,1]}$, 
$\Phi_f$ induces a self-homeomorphism 
$$\underline{\Phi}_f: S^2 \times [-1,1] \rightarrow S^2 \times [-1,1].$$
Since $ \Phi_f(\mathrm{cl}(\widetilde{b}))= \mathrm{cl}(\widetilde{b_f})$ we have 
$\underline{\Phi}_f (C(b))= C(b_f)$. 
Any orientation-preserving self-homeomorphism 
on $S^2$ is isotopic to the identity, 
and $S^3$ is a union of $S^2 \times [-1,1]$ and 
two $3$-balls by gluing the boundaries together. 
Thus $\underline{\Phi}_f$ extends to 
a self-homeomorphism on $S^3$ which sends    $C(b)$ to $C(b_f)$. 
This completes the proof.
\end{proof}

\begin{center}
\begin{figure}[t]
\includegraphics[height=4.5cm]{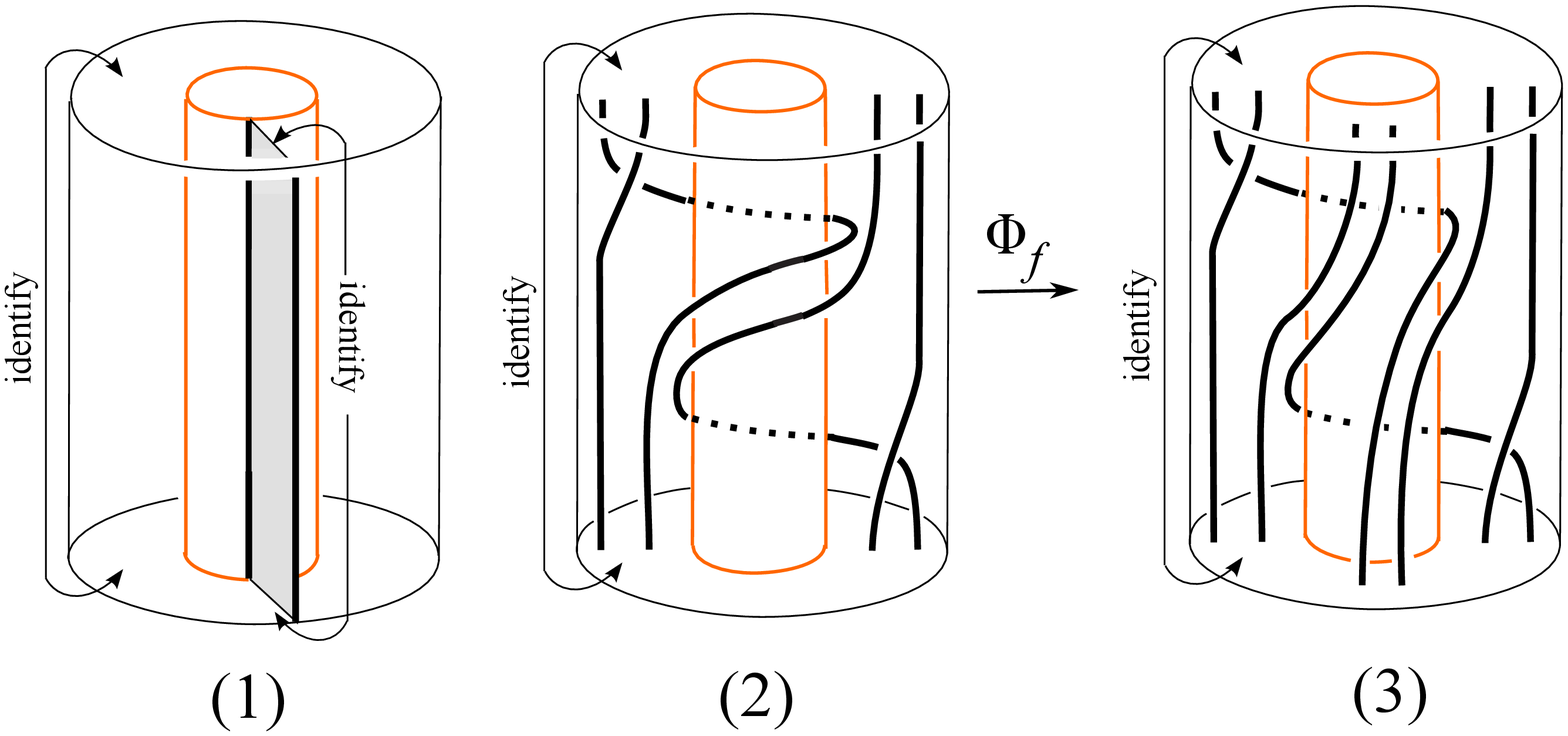}
\caption{
(1) Shaded region 
indicates  annulus ${\Bbb A}$ with boundary $\{v\} \times S^1 \times \{\pm 1\}$. 
(3) $\mathrm{cl}(\widetilde{b_1}) $ and 
(4) $\Phi_f(\mathrm{cl}(\widetilde{b_1})) =\mathrm{cl}(\widetilde{b_2})$ 
contained in $S^1 \times S^1 \times [-1,1]$.} 
\label{fig:disk-twist}
\end{figure}
\end{center}

Let us consider the mapping class group $\mathrm{Mod}(D_n)$ of the $n$-punctured disk $D_n$ 
preserving  the boundary $\partial D$ of the disk setwise. 
We have a surjective homomorphism 
$$\Gamma: B_n \rightarrow \mathrm{Mod}(D_n)$$
which sends each generator $\sigma_i$ to the right-handed half twist 
between the $i$th and $(i+1)$st punctures. 
We say that $\beta \in B_n$ is {\it pseudo-Anosov} if 
$\Gamma(\beta) \in  \mathrm{Mod}(D_n)$ is of the pseudo-Anosov type. 
When $\beta$ is a pseudo-Anosov braid, the dilatation $\lambda(\beta)$ is defined by the dilatation of 
$\lambda(\Gamma(\beta))$. 

We consider the above homomorphism 
$\Gamma: B_{2g+2} \rightarrow \mathrm{Mod}(D_{2g+2})$ when $n= 2g+2$. 
Recall the homomorphism 
$\mathfrak{t}: B_{2g+2} \rightarrow \mathrm{Mod}(\Sigma_g)$. 
The following lemma relates dilatations of $\beta $ and $\mathfrak{t}(\beta) $. 

\begin{lem}
\label{lem_disk}
Let $\beta \in B_{2g+2}$ be pseudo-Anosov and let 
$\Phi_{\beta} : D_{2g+2} \rightarrow D_{2g+2}$ be a pseudo-Anosov homeomorphism 
which represents $\Gamma(\beta) \in \mathrm{Mod}(D_{2g+2})$. 
Suppose that the stable foliation $\mathcal{F}_{\beta}$ for $\Phi_{\beta}$ defined on $D_{2g+2}$ 
is not $1$-pronged at the boundary $\partial D$ of the disk. 
Then $\mathfrak{t}(\beta) \in \mathrm{Mod}(\Sigma_g)$ is pseudo-Anosov, 
and $\lambda(\mathfrak{t}(\beta) ) = \lambda(\beta)$ holds. 
\end{lem}

\begin{proof}
Since $\mathcal{F}_{\beta}$ 
is not $1$-pronged at $\partial D$, 
$\Phi_{\beta} : D_{2g+2} \rightarrow D_{2g+2}$ induces a pseudo-Anosov 
homeomorphism 
$\Phi_{\beta}': \Sigma_{0,2g+2} \rightarrow \Sigma_{0,2g+2}$ 
by  filling  $\partial D$ with a disk. 
By Birman-Hilden~\cite{BH71}, 
we have  a surjective homomorphism 
$$\mathfrak{q}: \mathcal{H}(\Sigma_g)  \rightarrow \mathrm{Mod} (\Sigma_{0, 2g+2})$$ 
sending the Dehn twist $t_i$ to the right-handed half twist $h_i$ between the $i$th and $(i+1)$st punctures 
for $i= 1, \cdots, 2g+1$. 
Consider the $2$-fold branched cover $\Sigma_g \rightarrow \Sigma_{0, 2g+2}$ 
branched at the $2g+2$ marked points (corresponding to the punctures of $\Sigma_{0, 2g+2}$).  
Then there is a lift 
$f_{\beta}: \Sigma_g \rightarrow \Sigma_g$ of  $\Phi_{\beta}': \Sigma_{0,2g+2} \rightarrow \Sigma_{0,2g+2}$ 
which satisfies 
$\mathfrak{t}(\beta) = [f_{\beta}] \in \mathcal{H}(\Sigma_g)$. 
Note that 
the stable foliation $\mathcal{F}_{\beta}$ for $\Phi_{\beta}$ extends to the stable foliation $\mathcal{F}'_{\beta}$ 
for $\Phi'_{\beta}$ 
by the assumption that $\mathcal{F}_{\beta}$ 
is not $1$-pronged at $\partial{D}$. 
The stable foliation $\mathcal{F}'_{\beta}$ defined on $\Sigma_{0,2g+2}$ is lifted to the stable foliation for $f_{\beta}$ 
defined on $\Sigma_g$.  
Thus $f_{\beta}$ is a pseudo-Anosov homeomorphism which represents 
$\mathfrak{t}(\beta) = [f_{\beta}]$, and 
we have 
$$\lambda([f_{\beta}]) = \lambda([\Phi_{\beta}']) = \lambda([\Phi_{\beta}]) = \lambda(\beta).$$
This completes the proof. 
\end{proof}

\begin{center}
\begin{figure}[t]
\includegraphics[height=4.5cm]{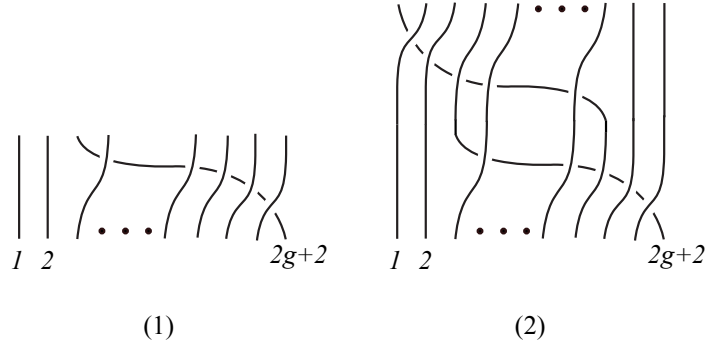}
\caption{
(1) $b_g \in B_{2g+2}$. 
(2) $\widetilde{b_g} \in B_{2g+2}$.}
\label{fig:122334_sakkuma}
\end{figure}
\end{center}

\begin{proof}[Proof of Theorem~\ref{thm_main}] 
For $g \ge 1$, we consider $b_g = \sigma_3 \sigma_4 \cdots \sigma_{2g+1} \in B_{2g+2}$ and 
$$\widetilde{b_g}= \sigma_1 \sigma_2 \cdots \sigma_{2g-1} \cdot \sigma_3 \sigma_4 \cdots \sigma_{2g+1} \in B_{2g+2},$$ 
see Figure~\ref{fig:122334_sakkuma}. 
By Penner's result 
it is enough to prove that 
$\mathfrak{t}(\widetilde{b_g}) $ is a pseudo-Anosov element in $D_g(S^3)$ for large $g$ 
and $\log \lambda(\mathfrak{t}(\widetilde{b_g})) \asymp 1/g$ holds.

Applying Theorem~\ref{thm:2-fold} for the braid $b_g $ 
we have the $2$-fold branched cover 
$$T_{\mathfrak{t}(\widetilde{b_g})} \rightarrow M_{C(b_g)}$$ 
branched over $q^{-1}(W)$. 
We first prove that $M_{C(b_g)} \cong S^3$ for $g \ge 1$. 
Clearly $C(b_1) = C(\sigma_3)$ is a trivial knot, 
and hence $M_{C(b_1)} \cong S^3$. 
We add a strand to $b_1= \sigma_3 \in B_4$ so that $b_1^{\circ} = \sigma_3^2  \sigma_4 \in B_5$, 
and  think of $\widetilde{b_1^{\circ}}$ as a braid on $\mathcal{A}$, see Figure~\ref{fig:adding}.   
Choose any $v \in S^1$ and consider the annulus ${\Bbb A}$ in $S^1 \times S^1 \times [-1,1]$ 
with boundary  $\{v\} \times S^1 \times \{\pm 1\}$, 
see Figure~\ref{fig:disk-twist}(1). 
Let $f: S^1 \times S^1 \rightarrow S^1 \times S^1$ be the Dehn twist on $S^1 \times S^1$  about $\{ v \} \times S^1$. 
Then the self-homeomorphism $\Phi_f$ on $S^1 \times S^1 \times [-1, 1]$
is an  annulus twist about ${\Bbb A}$, see Figure~\ref{fig:disk-twist}(2)(3). 
Observe that 
$\Phi_f (\mathrm{cl}(\widetilde{b_1})) = \mathrm{cl}(\widetilde{b_2})$ 
satisfies the assumptions $(*)$ and $(**)$. 
By repeating this process it is not hard to see that 
$$\Phi_f (\mathrm{cl}(\widetilde{b_{j-1}})) = \mathrm{cl}(\widetilde{b_j}) \hspace{2mm}
\mbox{for each } j \ge 2. $$
Thus $\Phi_f (\mathrm{cl}(\widetilde{b_{j-1}}))$ satisfies  $(*)$ and $(**)$ for each $j$.  
Lemma~\ref{lem_disk-twist-inv} tells us that 
$C(b_g)$ is  a trivial knot for $g \ge 1$ since so is $C(b_1)$. 
Thus $M_{C(b_g)} \cong S^3$, and 
$\mathfrak{t}(\widetilde{b_g}) \in D_g(S^3)$ for $g \ge 1$ by Theorem~\ref{thm:2-fold}.

The proof of Theorem~D in \cite{HK18} 
says that 
$\Gamma(\widetilde{b_g}) \in \mathrm{Mod}(D_{2g+2})$ is pseudo-Anosov for $g \ge 2$ 
and $\log \lambda(\widetilde{b_g}) \asymp \dfrac{1}{g}$ holds. 
Moreover the stable foliation of the pseudo-Anosov reprerentative for $\Gamma(\widetilde{b_g})$ 
satisfies the assumption of Lemma~\ref{lem_disk}, 
see the proof of Step 2 in \cite[Theorem~D]{HK18}. 
Thus 
$\mathfrak{t}(\widetilde{b_g}) $ is pseudo-Anosov with 
$\lambda(\mathfrak{t}(\widetilde{b_g})) = \lambda(\widetilde{b_g})$ by Lemma~\ref{lem_disk}, 
and  we obtain the desired claim 
$\log \lambda(\mathfrak{t}(\widetilde{b_g})) = \log \lambda(\widetilde{b_g}) \asymp \dfrac{1}{g}$. 
This completes the proof. 
\end{proof}

We end this paper with a question. 

\begin{ques}
Let $M$ be a closed $3$-manifold $M$ 
which is the $2$-fold branched cover of $S^3$ branched over some link. 
Then does it hold $\log \delta(D_g(M)) \asymp \dfrac{1}{g}$? 
\end{ques}


\begin{thebibliography}{99}

\bibitem{AY81}
P.~Arnoux and J-P.~Yoccoz, 
{\it Construction de diff\'eomorphismes pseudo-Anosov},
C. R. Acad. Sci. Paris S\'er. I Math. 292 
 (1981), no. 1, 75--78. 



\bibitem{BH71}
J.~Birman and H.~Hilden, 
{\it On mapping class groups of closed surfaces as cover spaces}, 
Advances in the theory of Riemann surfaces, 
Annals of Math Studies 66, Princeton University Press
(1971), 81-115. 

\bibitem{Brooks85} 
R.~Brooks, 
{\it On branched covers of $3$-manifolds which fiber over the circle}, 
J. Reine Angew. Math. 362 (1985), 87-101.

\bibitem{FM12} 
B.~Farb and D.~Margalit, 
A primer on mapping class groups, 
Princeton Mathematical Series 49, Princeton University Press, 
Princeton, NJ (2012). 


\bibitem{HK18}
S.~Hirose and E.~Kin, 
{\it A construction of pseudo-Anosov braids with small normalized entropies}, 
preprint (2018). 


\bibitem{Montesinos87}
J.~M.~Montesinos, 
{\it On $3$-manifolds having surface bundles as branched covers}, 
Proc. Amer. Math. Soc. 101 (1987), no. 3, 555-558. 

\bibitem{Penner91}
R.~C.~Penner, 
{\it Bounds on least dilatations}, 
Proceedings of the American Mathematical Society 113 (1991), 443-450. 

\bibitem{Sakuma81}
M.~Sakuma, 
{\it Surface bundles over $S^1$ which are $2$-fold branched cyclic covers of $S^3$}, 
Math. Sem. Notes, Kobe Univ., 9 (1981) 159-180. 

\bibitem{Thurston88}
W.~P.~Thurston, 
{\it On the geometry and dynamics of diffeomorphisms of surfaces},
  Bull. Amer. Math. Soc. (N.S.) 19 (1988), no. 2, 417--431. 


\bibitem{Thurston98} 
W.~Thurston, 
{\it Hyperbolic structures on $3$-manifolds II: Surface groups and 
$3$-manifolds which fiber over the circle}, preprint, 
arXiv:math/9801045

\end{thebibliography}
\end{document}